\documentclass[11pt]{article}
\usepackage{mathrsfs}
\usepackage{amsfonts}
\usepackage{}
\usepackage{amsmath,amssymb,amsthm,latexsym,amstext}
\usepackage[mathscr]{eucal}
\usepackage{rotate,graphics,epsfig,epstopdf}
\usepackage{float}
\usepackage{color}
\usepackage{subfigure}
\usepackage{indentfirst}
\usepackage{bm}

\newcommand{\be}{\begin{eqnarray}}
\newcommand{\ee}{\end{eqnarray}}
\newcommand{\by}{\begin{eqnarray*}}
\newcommand{\ey}{\end{eqnarray*}}
\newcommand{\bn}{\begin{enumerate}}
\newcommand{\en}{\end{enumerate}}
\newcommand{\ei}{\end{itemize}}

\newtheorem{theorem}{Theorem}
\newtheorem{lemma}[theorem]{Lemma}

\newtheorem{proposition}[theorem]{Proposition}

\renewcommand{\theequation}{\arabic{section}.\arabic{equation}}

\numberwithin{equation}{section}

\begin{document}
\date{}
\title{\bf Convergence rate of the Smoluchowski-Kramers approximation for diffusions with jumps  \footnote{This work was supported by 
the Natural Science Foundation of Jiangsu Province, BK20230899 and
the National Natural Science Foundation of China, 11771207.
}}
\author{ Chungang Shi\\
\texttt{{\scriptsize School of Mathematics and Statistics, Nanjing University of Science and Technology,
Nanjing, 210023, P. R. China}}}\maketitle
\begin{abstract}
In the paper, the Kolmogorov distance is used to study the Smoluchowski-Kramers approximation for diffusions with jumps. The convergence rate is derived by Malliavin calculus.
\end{abstract}



\section{Introduction}
  \setcounter{equation}{0}
  \renewcommand{\theequation}
{1.\arabic{equation}}
The motion of a particle with mass $\epsilon$ in a fluid subjected to friction and stochastic external forces satisfies Newton's equations of motion which may be described as~\cite{Ne}
\begin{eqnarray}\label{equ:main}
\epsilon\ddot{X}_{t}^{\epsilon}=b(t,X_{t}^{\epsilon})-\alpha\dot{X}_{t}^{\epsilon}+\dot{B}_{t}, \quad X_{0}^{\epsilon}=x_{0}, \dot{X}_{0}^{\epsilon}=v_{0},
\end{eqnarray}
where $\alpha\dot{X}_{t}^{\epsilon}$ is a linear dissipation term and $\dot{B}_{t}$ is a stochastic force that is a white noise. Formally, let $\epsilon\to0$ on both sides of equation (\ref{equ:main}), $X_{t}^{\epsilon}$ may be approximated by the solution of the following first order equation
\begin{eqnarray*}
\alpha\dot{X}_{t}=b(t,X_{t})+\dot{B}_{t},\quad X_{0}=x_{0},
\end{eqnarray*}
in the sense that
\begin{eqnarray*}
\lim_{\epsilon\to0}\mathbb{P}\big\{\max_{0\leq t\leq T}\|X_{t}^{\epsilon}-X_{t}\|>\delta\big\}=0.
\end{eqnarray*}
This is called classical Smoluchowski-Kramers approximation~\cite{Sm,K}. There are lots of works on the Smoluchowski-Kramers approximation with Gaussian noise~\cite{CF1,CF3,HMVW,SMD}. Recently, Tan and Dung~\cite{TD} derive the an explicit Berry-Esseen error bound for the rate of convergence in the Kolmogorov distance based on the Malliavin calculus. Son~\cite{S} study a stochastic differential equation driven by the fractional Brownian motion and derive the convergence rate  by using the same method.

In the paper, we consider the following stochastic differential equation driven by the Brownian motion and Poisson jumps with finite intensity measure $\nu$,
\begin{eqnarray}\label{equ:1.1}
\epsilon\ddot{X}_{t}^{\epsilon}&=&b(t,X_{t}^{\epsilon})-\dot{X}_{t}^{\epsilon}+\sigma(t,X_{t}^{\epsilon})dB_{t}+\int_{\mathbb{R}_{0}}c(X_{t-}^{\epsilon},z)\tilde{N}(dt,dz),\label{equ:1.1}\\
X_{0}^{\epsilon}&=&x_{0},\quad \dot{X}_{0}^{\epsilon}=y_{0},\nonumber
\end{eqnarray}
where $\mathbb{R}_{0}=\mathbb{R}\setminus\{0\}$. $B_{t}$ is a standard Brownian motion and $\tilde{N}$ is a compensated Poisson random measure with the intensity measure $\nu$ which is independent of $B_{t}$. By the similar proof as zhang~\cite{Z}, we may show that
\begin{equation*}
\lim_{\epsilon\to0}\mathbb{P}\{\|X^{\epsilon}-X\|_{D(0,T,\mathbb{R}^{d})}>\epsilon\}=0,
\end{equation*}
where $D(0,T,\mathbb{R}^{d})$ denotes the Banach space of c${\rm \grave{a}}$dl${\rm \grave{a}}$g functions on $[0,T]$ with the supremum norm~\cite{CHZ} 
and $X$ satisfies the following stochastic differential equation
\begin{equation}\label{equ:1.11}
\dot{X}_{t}=b(t,X_{t})+\sigma(t,X_{t})dB_{t}+\int_{\mathbb{R}_{0}}c(X_{t-},z)\tilde{N}(dt,dz).
\end{equation}

Our main goal is to obtain the convergence rate in the Kolmogorov distance between the laws of $X_{t}^{\epsilon}$ and $X_{t}$. 
Based on the Malliavin calculus, we derive the following convergence result
\begin{eqnarray}
\sup_{x\in\mathbb{R}^{d}}|\mathbb{P}(X_{t}^{\epsilon}\leq x)-\mathbb{P}(X_{t}\leq x)|\leq Ct^{-1}\sqrt{\epsilon},
\end{eqnarray}
where $C$ is a positive constant.

The rest of the paper is organized as follows. In section 2, we give some assumptions and some basic concepts of Malliavin calculus that is our main technique to derive our main result. The section 3 gives the proof of the main result.

\section{Preliminary}
  \setcounter{equation}{0}
  \renewcommand{\theequation}
{2.\arabic{equation}}

Let $(\Omega, \mathcal{F}, \mathbb{P})$ be a complete probability space with filter $\mathcal{F}_{t}$ is generated by $W$ and $N$, that is $\mathcal{F}_{t}=\sigma\{W_{s},N((0,s]\times A): 0\leq s\leq t, A\in \mathcal{B}(\mathbb{R}_{0})\}$.
Let $\{p(t)\}_{t\geq0}$ be $\sigma$-finite, stationary, $\mathcal{F}_{t}$-adapted Poisson point process and take value in measurable space $(\mathbb{R}_{0},\mathcal{B}(\mathbb{R}_{0}))$. Define a Poisson random measure which is induced by $\{p(t)\}$
\begin{eqnarray*}
N((s_{1},s_{2}], U)=\Sigma_{s\in(s_{1},s_{2}]}I_{U}(p(s)),
\end{eqnarray*}
for any $U\in \mathcal{B}(\mathbb{R}_{0})$, and the corresponding compensated Poisson random measure is
\begin{eqnarray*}
\tilde{N}(dt,dx)=N(dt,dx)-\nu(dx)dt,
\end{eqnarray*}
where $\nu$ is the intensity measure of $N$. As the main tool of the paper is the Malliavin calculus, we give some basic elements on Malliavin calculus in this part~\cite{N}.
Let $B=(B_{t})_{t\geq0}$ be a one-dimensional Brownian motion which is independent of $N$. For $h\in L^{1}(\mathbb{R}\times\mathbb{R}_{0}, dt\nu(dz))$, we make the following notation
\begin{eqnarray*}
N(h)=\int_{\mathbb{R}_{+}}\int_{\mathbb{R}_{0}}h(t,z)N(dt,dz)
\end{eqnarray*}
and for $g\in L^{2}(0,T)$, we denote by $B(g)$ the Wiener integral
\begin{eqnarray*}
B(g)=\int_{0}^{T}g(t)dB(t).
\end{eqnarray*}
We denote by $C_{0}^{0,2}(\mathbb{R}_{+}\times\mathbb{R}_{0})$ the set of continuous functions $h:\mathbb{R}_{+}\times\mathbb{R}_{0}\to\mathbb{R}$ that are twice continuously differentiable on $\mathbb{R}_{0}$ with compact support. We consider the set $\mathcal{S}$ of cylindrical random variables of the form 
\begin{eqnarray*}
F=\phi(B(g_{1}),\cdots,B(g_{k}),N(h_{1}),\cdots,N(h_{m})),
\end{eqnarray*}
where $\phi\in C_{0}^{2}(\mathbb{R}^{m+k}), g_{i}\in L^{2}(\mathbb{R}_{+})$ and $h_{i}\in C_{0}^{0,2}(\mathbb{R}_{+}\times\mathbb{R}_{0})$. The set $\mathcal{S}$ is dense in $L^{2}(\Omega)$. If $F$ belongs to $\mathcal{S}$, we define the Malliavin derivatives $D^{B}F$ and $D^{N}F$, considering the random variables $N(h_{i})$ and $B(g_{i})$, respectively, as constants. That is, for almost all $(t,z)\in\mathbb{R}_{+}\times\mathbb{R}_{0}$,  
\begin{eqnarray*}
D_{t}^{B}F=\sum_{i=1}^{k}\frac{\partial\phi}{\partial x_{i}}(B(g_{1}),\cdots,B(g_{k}),N(h_{1}),\cdots,N(h_{m}))g_{i}(t)
\end{eqnarray*}
and
\begin{eqnarray*}
D_{t,z}^{N}F=\sum_{i=k+1}^{k+m}\frac{\partial\phi}{\partial x_{i}}(B(g_{1}),\cdots,B(g_{k}),N(h_{1}),\cdots,N(h_{m}))\partial_{z}h_{i}(t,z),
\end{eqnarray*}
for almost all $(t,z)\in \mathbb{R}_{+}\times\mathbb{R}_{0}$. The operators $D^{B}$ and $D^{N}$ are closable and they can be extended to the space $\mathbb{D}^{2}$ defined as the closure of $\mathcal{S}$ with respect to the seminorm 
\begin{eqnarray*}
|||F|||_{2}^{2}=\mathbb{E}|F|^{2}+\mathbb{E}\|D^{B}F\|_{H}^{2}+\mathbb{E}\|D^{N}F\|_{N}^{2}.
\end{eqnarray*}
Here $\|\cdot\|_{H}$ and $\|\cdot\|_{N}$ denote the norm as follows respectively
\begin{eqnarray*}
\|u\|_{H}^{2}=\int_{0}^{T}|u_{t}|^{2}dt,\quad \|v\|_{N}^{2}=\int_{\mathbb{R}_{+}}\int_{\mathbb{R}_{0}}|v(s,z)|^{2}N(ds,dz).
\end{eqnarray*}
We define the space $\mathbb{L}_{a}^{1,N}$ as the space of predictable processes $u=(u(t,z))_{t\geq0,z\in\mathbb{R}_{0}}$ such that, for almost all $(t,z)$, we have that $u(t,z)\in\mathbb{D}^{1,N},\partial_{z}u(t,z)$ exists, and 
\begin{eqnarray*}
\|u\|_{1,N}^{2}=\int_{\mathbb{R}_{+}}\int_{\mathbb{R}_{0}}[\mathbb{E}|u(s,z)|^{2}+\mathbb{E}|\partial_{z}u(s,z)|^{2}+\mathbb{E}\|Du(s,z)\|_{N}^{2}]\nu(dz)ds<\infty.
\end{eqnarray*}
Then we have the following result~\cite{N}
\begin{proposition}
If $u$ belongs to $\mathbb{L}_{a}^{1,N}$ then the stochastic integrals
\begin{eqnarray*}
I=\int_{\mathbb{R}_{+}}\int_{\mathbb{R}_{0}}u(t,z)\tilde{N}(dt,dz)
\end{eqnarray*}
belongs to $\mathbb{D}^{1,N}$ and for almost all $(\tau,\alpha)\in\mathbb{R}_{+}\times\mathbb{R}_{0}$,
\begin{eqnarray*}
D_{\tau,\alpha}^{N}I=\partial_{z}u(\tau,\alpha)+\int_{\mathbb{R}_{+}}\int_{\mathbb{R}_{0}}D_{\tau,\alpha}u(t,z)\tilde{N}(dt,dz).
\end{eqnarray*}
\end{proposition}
Similarly, in the case of Brownian motion~\cite{N1}, we have
\begin{eqnarray*}
D_{r}^{B}\Big(\int_{0}^{T}u(s)ds\Big)&=&\int_{r}^{T}D_{r}^{B}u(s)ds, \nonumber\\
D_{r}^{B}\Big(\int_{0}^{T}u(s)dB(s)\Big)&=&u(r)+\int_{r}^{T}D_{r}^{B}u(s)dB(s), \quad0\leq r\leq T.
\end{eqnarray*}

\begin{lemma}[\cite{A}]{\rm(Kunita's first inequality)}\label{KN}
For any $p\geq 2$, there exists a constant $C_{p}>0$ such that
\begin{eqnarray*}
\mathbb{E}\sup_{0\leq s\leq t}\Big|\int_{0}^{s}\int_{Z}H(\tau,z)\tilde{N}(d\tau,dz)\Big|^{p}&\leq& C_{p}\Big\{\mathbb{E}\Big[\Big(\int_{0}^{s}\int_{Z}|H(\tau,z)|^{2}\nu(dz)d\tau\Big)^{\frac{p}{2}}\Big]\\
&&+\mathbb{E}\int_{0}^{s}\int_{Z}|H(\tau,z)|^{p}\nu(dz)d\tau\Big\}.
\end{eqnarray*}
\end{lemma}

Next we give some assumptions on coefficients.

$(\mathbf{H_{1}})$ $b, \sigma: \mathbb{R_{+}}\times\mathbb{R}\to \mathbb{R}, c: \mathbb{R_{+}}\times\mathbb{R}\times\mathbb{R}_{0}\to \mathbb{R}$ are Lipschitz continuous and have linear growth, that is, for any $x, y\in\mathbb{R}$, $0\leq t\leq T$,
\begin{eqnarray*}
|b(t,x)-b(t,y)|^{2}+|\sigma(t,x)-\sigma(t,y)|^{2}\leq K|x-y|^{2},
\end{eqnarray*}
\begin{eqnarray*}
\int_{\mathbb{R}_{0}}|c(x,z)-c(y,z)|^{2}\nu(dz)\leq K|x-y|^{2},
\end{eqnarray*}
and 
\begin{eqnarray*}
|b(t,x)|^{2}+|\sigma(t,x)|^{2}+\int_{\mathbb{R}_{0}}|c(x,z)|^{2}\nu(dz)\leq K(1+|x|^{2}).
\end{eqnarray*}

$(\mathbf{H_{2}})$ $b, \sigma, c(\cdot,z)\in C^{2}( \mathbb{R}, \mathbb{R})$ for any $z\in \mathbb{R}_{0}$, and $b, \sigma$ have bounded derivatives. Furthermore, $c(\cdot,\cdot)$ satisfies, for any $x\in\mathbb{R}, p\geq 2$,
\begin{eqnarray*}
&&\int_{\mathbb{R}_{0}}|\partial_{x}^{k}c(x,z)|^{p}\nu(dz)\leq K,\quad k=1,2,\\
&&\int_{\mathbb{R}_{0}}|\partial_{x}c(x,z)-\partial_{x}c(y,z)|^{2}\nu(dz)\leq K|x-y|^{2},
\end{eqnarray*}
and 
$|\partial_{z}c(x,z)|\leq K(1+|x|)$, $|\partial_{z}c(x,z)-\partial_{z}c(y,z)|\leq K|x-y|$.

\section{Main results}
 \setcounter{equation}{0}
 \renewcommand{\theequation}
{3.\arabic{equation}}

Under the assumption $(\mathbf{H_{1}})$ and $(\mathbf{H_{2}})$, by the Theorem 6.2.3 of~\cite{A} or the theorem 11.4.2 of~\cite{N}, 
the equation (\ref{equ:1.1}) and 
(\ref{equ:1.11}) have a unique strong solution, that is
\begin{eqnarray}\label{Sr}
X_{t}^{\epsilon}&=&x_{0}+\int_{0}^{t}Y_{s}^{\epsilon}ds\nonumber\\
Y_{t}^{\epsilon}&=&y_{0}+\frac{1}{\epsilon}\int_{0}^{t}(b(s,X_{s}^{\epsilon})-Y_{s}^{\epsilon})ds+\frac{1}{\epsilon}\int_{0}^{t}\sigma(s,X_{s}^{\epsilon})dB_{s}\nonumber\\
&&+\frac{1}{\epsilon}\int_{0}^{t}\int_{\mathbb{R}_{0}}c(X_{s-}^{\epsilon},z)\tilde{N}(ds,dz),
\end{eqnarray}
and 
\begin{eqnarray}\label{equ:3.4}
X_{t}=x_{0}+\int_{0}^{t}b(s,X_{s})ds+\int_{0}^{t}\sigma(s,X_{s})dB_{s}+\int_{0}^{t}\int_{\mathbb{R}_{0}}c(X_{s-},z)\tilde{N}(ds,dz).
\end{eqnarray}
By (\ref{equ:1.1}), the constant variation formula and integration by parts~\cite{F}, the equation (\ref{Sr}) is equivalent to the following 
\begin{eqnarray*}\label{equ:3.5}
Y_{t}^{\epsilon}&=&e^{-\frac{1}{\epsilon}t}y_{0}+\frac{1}{\epsilon}\int_{0}^{t}e^{-\frac{1}{\epsilon}(t-s)}b(s,X_{s}^{\epsilon})ds+\frac{1}{\epsilon}\int_{0}^{t}e^{-\frac{1}{\epsilon}(t-s)}\sigma(s,X_{s}^{\epsilon})dB_{s}\nonumber\\
&&+\frac{1}{\epsilon}\int_{0}^{t}\int_{\mathbb{R}_{0}}e^{-\frac{1}{\epsilon}(t-s)}c(X_{s-},z)\tilde{N}(ds,dz),
\end{eqnarray*}
and
\begin{eqnarray}\label{Xequ:3.6}
X_{t}^{\epsilon}&=&x_{0}+\epsilon y_{0}(1-e^{-\frac{1}{\epsilon}t})+\int_{0}^{t}b(s,X_{s}^{\epsilon})ds-\int_{0}^{t}e^{-\frac{1}{\epsilon}(t-s)}b(s,X_{s}^{\epsilon})ds\nonumber\\
&&+\int_{0}^{t}\sigma(s,X_{s}^{\epsilon})dB_{s}-\int_{0}^{t}e^{-\frac{1}{\epsilon}(t-s)}\sigma(s,X_{s}^{\epsilon})dB_{s}\\
&&+\int_{0}^{t}\int_{\mathbb{R}_{0}}c(X_{s-}^{\epsilon},z)\tilde{N}(ds,dz)-\int_{0}^{t}\int_{\mathbb{R}_{0}}e^{-\frac{1}{\epsilon}(t-s)}c(X_{s-}^{\epsilon},z)\tilde{N}(ds,dz).\nonumber
\end{eqnarray}
Then the following result holds.
\begin{lemma}\label{lem:3.1}
Under the assumptions $(\mathbf{H_{1}})$ and $(\mathbf{H_{2}})$, for any $p\geq 2$ and $\epsilon\in(0,1)$, we have
\begin{eqnarray}
&&\mathbb{E}\sup_{0\leq t\leq T}|X_{t}^{\epsilon}|^{p}\leq C.\label{equ:b3.5}\\
&&\sup_{0\leq t\leq T}\mathbb{E}|X_{t}^{\epsilon}-X_{t}|^{p}\leq C\epsilon^{\frac{p}{2}}.
\end{eqnarray}
\end{lemma}
\begin{proof}
By (\ref{Xequ:3.6}), the H${\rm \ddot{o}}$lder inequality, $(\mathbf{H_{1}})$, the Burkh${\rm \ddot{o}}$lder-Davies-Gundy inequality, the Jensen inequality and the Kunita's first inequality, 
\begin{eqnarray}\label{equ:3.9}
&&\mathbb{E}\sup_{0\leq t\leq T}|X_{t}^{\epsilon}|^{p}\leq C_{p}|x_{0}|^{p}+C_{p}|y_{0}|^{p}+C_{p}\mathbb{E}\sup_{0\leq t\leq T}\Big|\int_{0}^{t}b(s,X_{s}^{\epsilon})ds\Big|^{p}\nonumber\\
&&+C_{p}\mathbb{E}\sup_{0\leq t\leq T}\Big|\int_{0}^{t}e^{-\frac{1}{\epsilon}(t-s)}b(s,X_{s}^{\epsilon})ds\Big|^{p}+C_{p}\mathbb{E}\sup_{0\leq t\leq T}\Big|\int_{0}^{t}\sigma(s,X_{s}^{\epsilon})dB_{s}\Big|^{p}\nonumber\\
&&+C_{p}\mathbb{E}\sup_{0\leq t\leq T}\Big|\int_{0}^{t}e^{-\frac{1}{\epsilon}(t-s)}\sigma(s,X_{s}^{\epsilon})dB_{s}\Big|^{p}+C_{p}\mathbb{E}\sup_{0\leq t\leq T}\Big|\int_{0}^{t}\int_{\mathbb{R}_{0}}c(X_{s-}^{\epsilon},z)\tilde{N}(ds,dz)\Big|^{p}\nonumber\\
&&+C_{p}\mathbb{E}\sup_{0\leq t\leq T}\Big|\int_{0}^{t}\int_{\mathbb{R}_{0}}e^{-\frac{1}{\epsilon}(t-s)}c(X_{s-}^{\epsilon},z)\tilde{N}(ds,dz)\Big|^{p}\nonumber\\
&&\leq C_{p}|x_{0}|^{p}+C_{p}|y_{0}|^{p}+C_{p}T^{p-1}\int_{0}^{T}\mathbb{E}|b(s,X_{s}^{\epsilon})|^{p}ds\nonumber\\
&&+C_{p}\mathbb{E}\sup_{0\leq t\leq T}\Big(\int_{0}^{t}e^{-\frac{1}{\epsilon}\frac{p}{p-1}(t-s)}ds\Big)^{p-1}\int_{0}^{t}|b(s,X_{s}^{\epsilon})|^{p}ds+C_{p,T}\int_{0}^{T}\mathbb{E}|\sigma(s,X_{s}^{\epsilon})|^{p}ds\nonumber\\
&&+C_{p}\Big\{\mathbb{E}\Big(\int_{0}^{T}\int_{\mathbb{R}_{0}}|c(X_{s-}^{\epsilon},z)|^{2}\nu(dz)ds\Big)^{\frac{p}{2}}+\mathbb{E}\int_{0}^{T}\int_{\mathbb{R}_{0}}|c(X_{s-}^{\epsilon},z)|^{p}\nu(dz)ds\Big\}\nonumber\\
&&\leq C_{p,T}+C_{p,T}\int_{0}^{T}\mathbb{E}\sup_{0\leq s\leq t}|X_{s}^{\epsilon}|^{p}dt,
\end{eqnarray}
and the Gronwall inequality yields
\begin{eqnarray*}
\mathbb{E}\sup_{0\leq t\leq T}|X_{t}^{\epsilon}|^{p}\leq C_{p,T}.
\end{eqnarray*}
Furthermore, By (\ref{equ:3.4}) and (\ref{Xequ:3.6}), 
\begin{eqnarray*}
&&X_{t}^{\epsilon}-X_{t}\nonumber\\
&&=\epsilon y_{0}(1-e^{-\frac{1}{\epsilon}t})+\int_{0}^{t}(b(s,X_{s}^{\epsilon})-b(s,X_{s}))ds+\int_{0}^{t}(\sigma(s,X_{s}^{\epsilon})-\sigma(s,X_{s}))dB_{s}\nonumber\\
&&+\int_{0}^{t}\int_{\mathbb{R}_{0}}(c(X_{s-}^{\epsilon},z)-c(X_{s-},z))\tilde{N}(ds,dz)-\int_{0}^{t}e^{-\frac{1}{\epsilon}(t-s)}b(s,X_{s}^{\epsilon})ds\nonumber\\
&&-\int_{0}^{t}e^{-\frac{1}{\epsilon}(t-s)}\sigma(s,X_{s}^{\epsilon})dB_{s}-\int_{0}^{t}\int_{\mathbb{R}_{0}}e^{-\frac{1}{\epsilon}(t-s)}c(X_{s-}^{\epsilon},z)\tilde{N}(ds,dz),
\end{eqnarray*}
then by the similar proof to (\ref{equ:3.9}),
\begin{eqnarray*}\label{equ:3.18}
&&\mathbb{E}|X_{t}^{\epsilon}-X_{t}|^{p}\leq C_{p}|\epsilon y_{0}|^{p}+C_{p,T}\int_{0}^{t}\mathbb{E}|b(s,X_{s}^{\epsilon})-b(s,X_{s})|^{p}ds\nonumber\\
&&+C_{p,T}\int_{0}^{t}\mathbb{E}|\sigma(s,X_{s}^{\epsilon})-\sigma(s,X_{s})|^{p}ds+C_{p}\Big\{\mathbb{E}\Big(\int_{0}^{t}\int_{\mathbb{R}_{0}}|X_{s}^{\epsilon}-X_{s}|^{2}ds\Big)^{\frac{p}{2}}+\mathbb{E}\int_{0}^{t}|X_{s}^{\epsilon}-X_{s}|^{p}ds\Big\}\nonumber\\
&&+C_{p}\mathbb{E}\Big(\int_{0}^{t}e^{-\frac{1}{\epsilon}(t-s)}|b(s,X_{s}^{\epsilon})|ds\Big)^{p}+C_{p}\mathbb{E}\Big(\int_{0}^{t}e^{-\frac{2}{\epsilon}(t-s)}|\sigma(s,X_{s}^{\epsilon})|^{2}ds\Big)^{\frac{p}{2}}\nonumber\\
&&+C_{p}\Big\{\mathbb{E}\Big(\int_{0}^{t}\int_{\mathbb{R}_{0}}e^{-\frac{2}{\epsilon}(t-s)}|c(X_{s-}^{\epsilon},z)|^{2}\nu(dz)ds\Big)^{\frac{p}{2}}+\mathbb{E}\int_{0}^{t}\int_{\mathbb{R}_{0}}e^{-\frac{p}{\epsilon}(t-s)}|c(X_{s-}^{\epsilon},z)|^{p}\nu(dz)ds\Big\}\nonumber\\
&&\leq C_{p,T}\epsilon^{\frac{p}{2}}+C_{p,T}\mathbb{E}\int_{0}^{t}|X_{s}^{\epsilon}-X_{s}|^{p}ds,
\end{eqnarray*}
Gronwall's inequality yields
\begin{eqnarray*}
\mathbb{E}|X_{t}^{\epsilon}-X_{t}|^{p}\leq C_{p,T}\epsilon^{\frac{p}{2}}.
\end{eqnarray*}
\end{proof}

Now we consider the Malliavin property of the solution to equation (\ref{equ:1.1}) and (\ref{equ:3.4}). 
\begin{proposition}
Under the assumption $(\mathbf{H_{1}})$ and $(\mathbf{H_{2}})$, the solution to equation (\ref{equ:3.4}) is Malliavin differentiable and satisfies
\begin{eqnarray}\label{equ:3.27}
D_{r}^{B}X_{t}&=&\sigma(r,X_{r})+\int_{r}^{t}b'(s,X_{s})D_{r}^{B}X_{s}ds+\int_{r}^{t}\sigma'(s,X_{s})D_{r}^{B}X_{s}dB_{s}\nonumber\\
&&+\int_{r}^{t}\int_{\mathbb{R}_{0}}\partial_{x}c(X_{s-},z)D_{r}^{B}X_{s}\tilde{N}(ds,dz),\quad r\leq t,\nonumber\\
D_{r}^{B}X_{t}&=&0,\quad r>t,
\end{eqnarray}
and
\begin{eqnarray}\label{equ:3.28}
D_{r,\xi}^{N}X_{t}&=&\partial_{z}c(X_{r},\xi)+\int_{r}^{t}b'(s,X_{s})D_{r,\xi}^{N}X_{s}ds+\int_{r}^{t}\sigma'(s,X_{s})D_{r,\xi}^{N}X_{s}dB_{s}\nonumber\\
&&+\int_{r}^{t}\int_{\mathbb{R}_{0}}\partial_{x}c(X_{s-},z)D_{r,\xi}^{N}X_{s}\tilde{N}(ds,dz),\quad r\leq t,\nonumber\\
D_{r,\xi}^{N}X_{t}&=&0,\quad r>t.
\end{eqnarray}
\end{proposition}
\begin{proof}
As the solution $\{X_{t}\}_{0\leq t\leq T}$ is $\mathcal{F}_{t}$-adapted, then 
\begin{eqnarray}
D_{r}^{B}X_{t}=D_{r,\xi}^{N}X_{t}=0, \quad t<r, \nonumber
\end{eqnarray}
while $t\geq r$, by Theorem 11.4.3 in~\cite{N} or Theorem 3 in~\cite{Pe}, we derive (\ref{equ:3.27}) and (\ref{equ:3.28}). 
\end{proof}
Similarly, 
\begin{proposition}
Under the assumption $(\mathbf{H_{1}})$ and $(\mathbf{H_{2}})$, the solution to equation (\ref{equ:1.1}) is Malliavin differentiable and satisfies
\begin{eqnarray}
D_{r}^{B}X_{t}^{\epsilon}&=&(1-e^{-\frac{\alpha}{\epsilon}(t-r)})\sigma(r,X_{r}^{\epsilon})+\int_{r}^{t}b'(s,X_{s}^{\epsilon})D_{r}^{B}X_{s}^{\epsilon}ds\nonumber\\
&&-\int_{0}^{t}e^{-\frac{1}{\epsilon}(t-s)}b'(s,X_{s}^{\epsilon})D_{r}^{B}X_{s}^{\epsilon}ds+\int_{r}^{t}\sigma'(s,X_{s}^{\epsilon})D_{r}^{B}X_{s}^{\epsilon}dB_{s}\nonumber\\
&&-\int_{r}^{t}e^{-\frac{1}{\epsilon}(t-s)}\sigma'(s,X_{s}^{\epsilon})D_{r}^{B}X_{s}^{\epsilon}dB_{s}+\int_{r}^{t}\int_{\mathbb{R}_{0}}\partial_{x}c(X_{s-}^{\epsilon},z)D_{r}^{B}X_{s}^{\epsilon}\tilde{N}(ds,dz)\nonumber\\
&&-\int_{r}^{t}\int_{\mathbb{R}_{0}}e^{-\frac{1}{\epsilon}(t-s)}\partial_{x}c(X_{s-}^{\epsilon},z)D_{r}^{B}X_{s}^{\epsilon}\tilde{N}(ds,dz),\quad r\leq t,\label{Xep}\\
D_{r}^{B}X_{t}^{\epsilon}&=&0,\quad r>t.\nonumber
\end{eqnarray}
and
\begin{eqnarray}
D_{r,\xi}^{N}X_{t}^{\epsilon}&=&(1-e^{-\frac{1}{\epsilon}(t-r)})\partial_{z}c(X_{r}^{\epsilon},\xi)+\int_{r}^{t}b'(s,X_{s}^{\epsilon})D_{r,\xi}^{N}X_{s}^{\epsilon}ds\nonumber\\
&&-\int_{r}^{t}e^{-\frac{1}{\epsilon}(t-s)}b'(s,X_{s}^{\epsilon})D_{r,\xi}^{N}X_{s}^{\epsilon}ds+\int_{r}^{t}\sigma'(s,X_{s}^{\epsilon})D_{r,\xi}^{N}X_{s}^{\epsilon}dB_{s}\nonumber\\
&&-\int_{r}^{t}e^{-\frac{1}{\epsilon}(t-s)}\sigma'(s,X_{s}^{\epsilon})D_{r,\xi}^{N}X_{s}^{\epsilon}dB_{s}+\int_{r}^{t}\int_{\mathbb{R}_{0}}\partial_{x}c(X_{s-}^{\epsilon},z)D_{r,\xi}^{N}X_{s}^{\epsilon}\tilde{N}(ds,dz)\nonumber\\
&&-\int_{r}^{t}\int_{\mathbb{R}_{0}}e^{-\frac{1}{\epsilon}(t-s)}\partial_{x}c(X_{s-}^{\epsilon},z)D_{r,\xi}^{N}X_{s}^{\epsilon}\tilde{N}(ds,dz),~~ r\leq t,\label{DRX}\\
D_{r,\xi}^{N}X_{t}^{\epsilon}&=&0,\quad r>t.\nonumber
\end{eqnarray}
\end{proposition}
\begin{lemma}\label{Le}
Under the assumption $(\mathbf{H_{1}})$ and $(\mathbf{H_{2}})$, then for any $p\geq2$ and $\epsilon\in(0,1)$, we have
\begin{eqnarray}
\mathbb{E}|D_{r}^{B}X_{t}^{\epsilon}|^{p}&\leq& C_{T},\quad 0\leq r\leq t\leq T,\label{equ:3.32}\\
\mathbb{E}|D_{r,\xi}^{N}X_{t}^{\epsilon}|^{p}&\leq& C_{T},\quad 0\leq r\leq t\leq T,\label{equ:3.33}
\end{eqnarray}
and
\begin{eqnarray}
\mathbb{E}\|D^{B}X_{t}^{\epsilon}-D^{B}X_{t}\|_{L^{2}(0,T)}^{2}&\leq& C_{T}\epsilon.\\\label{equ:3.34}
\mathbb{E}\|D^{N}X_{t}^{\epsilon}-D^{N}X_{t}\|_{L^{2}([0,T]\times\mathbb{R}_{0})}^{2}&\leq& C_{T}\epsilon,\label{equ:3.35}
\end{eqnarray}
where $\|f\|_{L^{2}([0,T]\times\mathbb{R}_{0})}=\int_{0}^{T}\int_{\mathbb{R}_{0}}|f(t,z)|^{2}\nu(dz)dt$.
\end{lemma}
\begin{proof}
By (\ref{Xep}) and $(\mathbf{H_{1}})$ and $(\mathbf{H_{2}})$, we have
\begin{eqnarray}
&&\mathbb{E}|D_{r}^{B}X_{t}^{\epsilon}|^{p}\leq C_{p}\mathbb{E}|\sigma(r,X_{r}^{\epsilon})|^{p}+C_{p}\mathbb{E}\Big|\int_{r}^{t}b'(s,X_{s}^{\epsilon})D_{r}^{B}X_{s}^{\epsilon}ds\Big|^{p}\nonumber
\end{eqnarray}
\begin{eqnarray}\label{equ:3.36}
&&+C_{p}\mathbb{E}\Big|\int_{0}^{t}e^{-\frac{1}{\epsilon}(t-s)}b'(s,X_{s}^{\epsilon})D_{r}^{B}X_{s}^{\epsilon}ds\Big|^{p}+C_{p}\mathbb{E}\Big|\int_{t}^{r}\sigma'(s,X_{s}^{\epsilon})D_{r}^{B}X_{s}^{\epsilon}dB_{s}\Big|^{p}\nonumber\\
&&+C_{p}\mathbb{E}\Big|\int_{t}^{r}e^{-\frac{1}{\epsilon}(t-s)}\sigma'(s,X_{s}^{\epsilon})D_{r}^{B}X_{s}^{\epsilon}dB_{s}\Big|^{p}+C_{p}\mathbb{E}\Big|\int_{r}^{t}\int_{\mathbb{R}_{0}}\partial_{x}c(X_{s-}^{\epsilon},z)D_{r}^{B}X_{s}^{\epsilon}\tilde{N}(ds,dz)\Big|^{p}\nonumber\\
&&+C_{p}\mathbb{E}\Big|\int_{r}^{t}\int_{\mathbb{R}_{0}}e^{-\frac{1}{\epsilon}(t-s)}\partial_{x}c(X_{s-}^{\epsilon},z)D_{r}^{B}X_{s}^{\epsilon}\tilde{N}(ds,dz)\Big|^{p}\nonumber\\
&&\triangleq\sum_{i=1}^{7}Q_{i}^{\epsilon}(t).
\end{eqnarray}
It is obvious that
\begin{eqnarray}\label{equ:3.37}
Q_{1}^{\epsilon}(t)\leq C_{p}+C_{p}\mathbb{E}|X_{r}^{\epsilon}|^{p}.
\end{eqnarray}
By $(\mathbf{H_{2}})$ and the Burkh${\rm \ddot{o}}$lder-Davies-Gundy inequality, 
\begin{eqnarray}\label{equ:3.38}
Q_{i}^{\epsilon}(t)\leq C_{p,T}\int_{r}^{t}\mathbb{E}|D_{r}^{B}X_{s}^{\epsilon}|^{p}ds,\quad i=2,3,4,5.
\end{eqnarray}
By Kunita's first inequality and $(\mathbf{H_{2}})$, we have
\begin{eqnarray}\label{equ:3.39}
Q_{6}^{\epsilon}(t)&\leq& C_{p}\Big\{\mathbb{E}\Big(\int_{r}^{t}\int_{\mathbb{R}_{0}}|\partial_{x}c(X_{s-}^{\epsilon},z)D_{r}^{B}X_{s}^{\epsilon}|^{2}\nu(dz)ds\Big)^{\frac{p}{2}}\nonumber\\
&&+\mathbb{E}\int_{r}^{t}\int_{\mathbb{R}_{0}}|\partial_{x}c(X_{s-}^{\epsilon},z)|^{p}|D_{r}^{B}X_{s}^{\epsilon}|^{p}\nu(dz)ds\Big\}\nonumber\\
&\leq&C_{p}\Big\{\mathbb{E}\Big(\int_{r}^{t}|D_{r}^{B}X_{s}^{\epsilon}|^{2}ds\Big)^{\frac{p}{2}}+\mathbb{E}\int_{r}^{t}|D_{r}^{B}X_{s}^{\epsilon}|^{p}ds\Big\}\nonumber\\
&\leq& C_{p,T}\mathbb{E}\int_{r}^{t}|D_{r}^{B}X_{s}^{\epsilon}|^{p}ds.
\end{eqnarray}
Similarly, 
\begin{eqnarray}\label{equ:3.40}
Q_{7}^{\epsilon}(t)&\leq&C_{p,T}\mathbb{E}\int_{r}^{t}|D_{r}^{B}X_{s}^{\epsilon}|^{p}ds.
\end{eqnarray}
Then by (\ref{equ:3.36}) and (\ref{equ:3.37})-(\ref{equ:3.40}), 
\begin{eqnarray*}
\mathbb{E}|D_{r}^{B}X_{t}^{\epsilon}|^{p}&\leq&C_{p}+C_{p}\mathbb{E}|X_{r}^{\epsilon}|^{p}+C_{p,T}\mathbb{E}\int_{r}^{t}|D_{r}^{B}X_{s}^{\epsilon}|^{p}ds.
\end{eqnarray*}
Gronwall's inequality and (\ref{equ:b3.5}) yield
\begin{eqnarray*}\label{DX}
\mathbb{E}|D_{r}^{B}X_{t}^{\epsilon}|^{p}&\leq&C_{p,T}.
\end{eqnarray*}
By (\ref{equ:3.27}) and (\ref{Xep}), 
\begin{eqnarray}\label{equ:3.42}
&&D_{r}^{B}X_{t}^{\epsilon}-D_{r}^{B}X_{t}=[(1-e^{-\frac{1}{\epsilon}(t-r)})\sigma(r,X_{r}^{\epsilon})-\sigma(r,X_{r})]\nonumber\\
&&-\int_{r}^{t}e^{-\frac{1}{\epsilon}(t-s)}b'(s,X_{s}^{\epsilon})D_{r}^{B}X_{s}^{\epsilon}ds-\int_{r}^{t}\int_{\mathbb{R}_{0}}e^{-\frac{1}{\epsilon}(t-s)}\partial_{x}c(X_{s-}^{\epsilon},z)D_{r}^{B}X_{s}^{\epsilon}\tilde{N}(ds,dz)\nonumber\\
&&+\int_{r}^{t}[b'(s,X_{s}^{\epsilon})D_{r}^{B}X_{s}^{\epsilon}-b'(s,X_{s})D_{r}^{B}X_{s}]ds\nonumber\\
&&+\int_{r}^{t}[\sigma'(s,X_{s}^{\epsilon})D_{r}^{B}X_{s}^{\epsilon}-\sigma'(s,X_{s})D_{r}^{B}X_{s}]dB_{s}\nonumber\\
&&+\int_{r}^{t}\int_{\mathbb{R}_{0}}[\partial_{x}c(X_{s-}^{\epsilon},z)D_{r}^{B}X_{s}^{\epsilon}-\partial_{x}c(X_{s},z)D_{r}^{B}X_{s}]\tilde{N}(ds,dz).
\end{eqnarray}
By (\ref{equ:b3.5}), for any $0\leq r\leq t\leq T$,
\begin{eqnarray}\label{equ:3.43}
\mathbb{E}|(1-e^{-\frac{1}{\epsilon}(t-r)})\sigma(r,X_{r}^{\epsilon})-\sigma(r,X_{r})|^{2}
&\leq& 2K^{2}\mathbb{E}|X_{r}^{\epsilon}-X_{r}|^{2}+2K^{2}e^{-\frac{2}{\epsilon}(t-r)}(1+\mathbb{E}|X_{r}^{\epsilon}|^{2})\nonumber\\
&\leq& C\epsilon+Ce^{-\frac{2}{\epsilon}(t-r)}.
\end{eqnarray}
Furthermore, by the H${\rm \ddot{o}}$lder inequality, the Burkh${\rm \ddot{o}}$lder-Davies-Gundy inequality, (\ref{equ:3.32}) and the boundedness of $b'$ as well as $\sigma'$, 
\begin{eqnarray}\label{equ:3.44}
\mathbb{E}\Big|\int_{r}^{t}e^{-\frac{1}{\epsilon}(t-s)}b'(s,X_{s}^{\epsilon})D_{r}^{B}X_{s}^{\epsilon}ds\Big|^{2}
&\leq& C_{T}\int_{r}^{t}e^{-\frac{1}{\epsilon}(t-s)}ds\leq C_{T}\epsilon,
\end{eqnarray} 
and 
\begin{eqnarray}\label{equ:3.45}
\mathbb{E}\Big|\int_{t}^{r}e^{-\frac{1}{\epsilon}(t-s)}\sigma'(s,X_{s}^{\epsilon})D_{r}^{B}X_{s}^{\epsilon}dB_{s}\Big|^{2}&\leq& C_{T}\int_{r}^{t}e^{-\frac{2}{\epsilon}(t-s)}\mathbb{E}|D_{r}^{B}X_{s}^{\epsilon}|^{2}ds\nonumber\\
&\leq&C_{T}\int_{r}^{t}e^{-\frac{2}{\epsilon}(t-s)}ds\leq C_{T}\epsilon,
\end{eqnarray}
For the fourth term in (\ref{equ:3.42}), by the Burkh${\rm \ddot{o}}$lder-Davies-Gundy inequality, $(\mathbf{H_{2}})$ and (\ref{equ:3.32}), 
\begin{eqnarray}\label{equ:3.46}
&&\mathbb{E}\Big|\int_{r}^{t}\int_{\mathbb{R}_{0}}e^{-\frac{2}{\epsilon}(t-s)}\partial_{x}c(X_{s-}^{\epsilon},z)D_{r}^{B}X_{s}^{\epsilon}\tilde{N}(ds,dz)\Big|^{2}\nonumber\\
&\leq&C\mathbb{E}\int_{r}^{t}\int_{\mathbb{R}_{0}}e^{-\frac{2}{\epsilon}(t-s)}|\partial_{x}c(X_{s-}^{\epsilon},z)|^{2}|D_{r}^{B}X_{s}^{\epsilon}|^{2}\nu(dz)ds\nonumber\\
&\leq&C\int_{r}^{t}e^{-\frac{2}{\epsilon}(t-s)}ds\leq C\epsilon,\quad 0\leq r\leq t\leq T.
\end{eqnarray}
For the fifth term in (\ref{equ:3.42}), by the H${\rm \ddot{o}}$lder inequality, the boundedness of $b'$ and $b''$ and lemma \ref{lem:3.1}, 
\begin{eqnarray}
&&\mathbb{E}\Big|\int_{r}^{t}[b'(s,X_{s}^{\epsilon})D_{r}^{B}X_{s}^{\epsilon}-b'(s,X_{s})D_{r}^{B}X_{s}]ds\Big|^{2}\nonumber
\end{eqnarray}
\begin{eqnarray}\label{equ:3.47}
&\leq&C_{T}\int_{r}^{t}(\mathbb{E}|X_{s}^{\epsilon}-X_{s}|^{4})^{\frac{1}{2}}(\mathbb{E}|D_{r}^{B}X_{s}^{\epsilon}|^{4})^{\frac{1}{2}}ds+C_{T}\int_{r}^{t}\mathbb{E}|D_{r}^{B}X_{s}^{\epsilon}-D_{r}^{B}X_{s}|^{2}ds\nonumber\\
&\leq& C\epsilon+C_{T}\int_{r}^{t}\mathbb{E}|D_{r}^{B}X_{s}^{\epsilon}-D_{r}^{B}X_{s}|^{2}ds.
\end{eqnarray}
Similarly,
\begin{eqnarray}\label{equ:3.48}
&&\mathbb{E}\Big|\int_{r}^{t}[\sigma'(s,X_{s}^{\epsilon})D_{r}^{B}X_{s}^{\epsilon}-\sigma'(s,X_{s})D_{r}^{B}X_{s}]dB_{s}\Big|^{2}\nonumber\\
&&\leq C\epsilon+C\int_{r}^{t}\mathbb{E}|D_{r}^{B}X_{s}^{\epsilon}-D_{r}^{B}X_{s}|^{2}ds,
\end{eqnarray}
and
\begin{eqnarray}\label{equ:3.49}
&&\mathbb{E}\Big|\int_{r}^{t}\int_{\mathbb{R}_{0}}[\partial_{x}c(X_{s-}^{\epsilon},z)D_{r}^{B}X_{s}^{\epsilon}-\partial_{x}c(X_{s},z)D_{r}^{B}X_{s}]\tilde{N}(ds,dz)\Big|^{2}\nonumber\\
&\leq&2K\int_{r}^{t}\mathbb{E}|X_{s}^{\epsilon}-X_{s}|^{2}|D_{r}^{B}X_{s}^{\epsilon}|^{2}ds+2K\int_{r}^{t}\mathbb{E}|D_{r}^{B}X_{s}^{\epsilon}-D_{r}^{B}X_{s}|^{2}ds\nonumber\\
&\leq&C\epsilon+C_{T}\int_{r}^{t}\mathbb{E}|D_{r}^{B}X_{s}^{\epsilon}-D_{r}^{B}X_{s}|^{2}ds.
\end{eqnarray}
By (\ref{equ:3.42})-(\ref{equ:3.49}), we obtain
\begin{eqnarray}
\mathbb{E}|D_{r}^{B}X_{t}^{\epsilon}-D_{r}^{B}X_{t}|^{2}\leq C\epsilon+Ce^{-\frac{2}{\epsilon}(t-r)}+C_{T}\int_{r}^{t}\mathbb{E}|D_{r}^{B}X_{s}^{\epsilon}-D_{r}^{B}X_{s}|^{2}ds.\nonumber
\end{eqnarray}
Then
\begin{eqnarray}
&&\mathbb{E}\|D^{B}X_{t}^{\epsilon}-D^{B}X_{t}\|_{L^{2}(0,T)}^{2}=\int_{0}^{T}\mathbb{E}|D_{r}^{B}X_{t}^{\epsilon}-D_{r}^{B}X_{t}|^{2}dr\nonumber\\
&&\leq C\epsilon+C\int_{0}^{t}\mathbb{E}\|D^{B}X_{s}^{\epsilon}-D^{B}X_{s}\|_{L^{2}(0,T)}^{2}ds,
\end{eqnarray}
Gronwall's inequality yields
\begin{eqnarray}
\mathbb{E}\|D^{B}X_{s}^{\epsilon}-D^{B}X_{s}\|_{L^{2}(0,T)}^{2}&\leq&C\epsilon.
\end{eqnarray}
A similar proof leads to (\ref{equ:3.33}) and (\ref{equ:3.35}).
\end{proof}
In the following, if there's no confusion, let's write $\|\cdot\|_{L^{2}(0,T)}$ as $\|\cdot\|$.
\begin{lemma}\label{lem:3.3}
Under the assumption $(\mathbf{H_{1}})$ and $(\mathbf{H_{2}})$, let $\|\sigma\|_{0}\\
=\inf_{(t,x)\in[0,T]\times\mathbb{R}}|\sigma(t,x)|>0$ and $\|\partial_{z}c\|_{0}=\inf_{(x,z)\in\mathbb{R}\times\mathbb{R}_{0}}|\partial_{z}c(x,z)|>0$, then for any $p\geq1$,
\begin{eqnarray}
\mathbb{E}\Big[\frac{1}{\|D^{B}X_{t}\|^{2p}}\Big]\leq Ct^{-p},\quad 0\leq t\leq T,
\end{eqnarray}
and 
\begin{eqnarray}
\mathbb{E}\Big[\frac{1}{\|D^{N}X_{t}\|^{2p}}\Big]\leq Ct^{-p},\quad 0\leq t\leq T.\label{equ:b3.53}
\end{eqnarray}
\end{lemma}
\begin{proof}
Note that equation (\ref{equ:3.27}) is linear with respect to $D_{r}^{B}X_{t}$, by It${\rm \hat{o}}$'s formula~(see e.g. Theorem 9.5.2 in~\cite{N}), for any $0\leq r\leq t\leq T$, we have
\begin{eqnarray}\label{equ:3.52}
D_{r}^{B}X_{t}&=&\sigma(r,X_{r})\exp\Big(\int_{r}^{t}(b'(s,X_{s})-\frac{1}{2}(\sigma'(s,X_{s}))^{2})ds+\int_{r}^{t}\sigma'(s,X_{s})dB_{s}\nonumber\\
&&+\int_{r}^{t}\int_{\mathbb{R}_{0}}\log(1+\partial_{x}c(X_{s-},z))\tilde{N}(ds,dz)\nonumber\\
&&+\int_{r}^{t}\int_{\mathbb{R}_{0}}(\log(1+\partial_{x}c(X_{s-},z))-\partial_{x}c(X_{s-},z))\nu(dz)ds\Big),
\end{eqnarray}
and
\begin{eqnarray}\label{equ:3.53}
D_{r,\xi}^{N}X_{t}&=&\partial_{z}c(X_{r},\xi)\exp\Big(\int_{r}^{t}\sigma'(s,X_{s})dB_{s}+\int_{r}^{t}(b'(s,X_{s})-\frac{1}{2}(\sigma'(s,X_{s}))^{2})ds\nonumber\\
&&+\int_{r}^{t}\int_{\mathbb{R}_{0}}\log(1+\partial_{x}c(X_{s-},z))\tilde{N}(ds,dz)\nonumber\\
&&+\int_{r}^{t}\int_{\mathbb{R}_{0}}(\log(1+\partial_{x}c(X_{s-},z))-\partial_{x}c(X_{s-},z))\nu(dz)ds\Big).
\end{eqnarray}
By (\ref{equ:3.52}), we have
\begin{eqnarray}\label{equ:3.54}
&&|D_{r}^{B}X_{t}|^{2}\geq \|\sigma\|_{0}^{2}\exp\Big(\int_{r}^{t}(2b'(s,X_{s})-(\sigma'(s,X_{s}))^{2})ds\nonumber
\\
&&+2\int_{r}^{t}\int_{\mathbb{R}_{0}}(\log(1+\partial_{x}c(X_{s-},z))-\partial_{x}c(X_{s-},z))\nu(dz)ds\Big)\cdot\exp\Big(2\int_{r}^{t}\sigma'(s,X_{s})dB_{s}\Big)\nonumber\\
&&\cdot\exp\Big(2\int_{r}^{t}\int_{\mathbb{R}_{0}}\log(1+\partial_{x}c(X_{s-},z))\tilde{N}(ds,dz)\Big)
\end{eqnarray}
Note that the fact: $\log(1+x)-x\leq Lx^{2}$, here $L>0$ is a constant, and together with $(\mathbf{H_{2}})$, we obtain
\begin{eqnarray}\label{equ:3.55}
\int_{r}^{t}\int_{\mathbb{R}_{0}}(\log(1+\partial_{x}c(X_{s-},z))-\partial_{x}c(X_{s-},z))\nu(dz)ds&\leq&L\int_{r}^{t}\int_{\mathbb{R}_{0}}|\partial_{x}c(X_{s-},z)|^{2}\nu(dz)ds\nonumber\\
&\leq& LKT.
\end{eqnarray}
Then by (\ref{equ:3.54}) and (\ref{equ:3.55}), we have
\begin{eqnarray*}
|D_{r}^{B}X_{t}|^{2}&\geq&\|\sigma\|_{0}^{2}e^{(-2K-K^{2})T}e^{-2LKT}\exp\Big(2\int_{r}^{t}\sigma'(s,X_{s})dB_{s}\nonumber\\&&+2\int_{r}^{t}\int_{\mathbb{R}_{0}}\log(1+\partial_{x}c(X_{s-},z))\tilde{N}(ds,dz)\Big).
\end{eqnarray*}
Define $M_{t}=2\int_{0}^{t}\sigma'(s,X_{s})dB_{s}+2\int_{0}^{t}\int_{\mathbb{R}_{0}}\log(1+\partial_{x}c(X_{s-},z))\tilde{N}(ds,dz)$,
then
\begin{eqnarray*}
\|D^{B}X_{t}\|^{2}&\geq&\|\sigma\|_{0}^{2}e^{(-2K-K^{2})T}e^{-2LKT}\int_{0}^{t}e^{2M_{t}-2M_{r}}dr\nonumber\\
&\geq&\|\sigma\|_{0}^{2}e^{(-2K-K^{2})T}e^{-2LKT}e^{4\min_{0\leq t\leq T}M_{t}}\cdot t.
\end{eqnarray*}
Note that $M_{t}$ is a martingale with bounded quadratic variation. In fact, 
\begin{eqnarray*}
\langle M\rangle_{t}=\int_{0}^{t}|\sigma'(s,X_{s})|^{2}ds+\int_{0}^{t}\int_{\mathbb{R}_{0}}|\log(1+\partial_{x}c(X_{s-},z))|^{2}N(dz,ds),
\end{eqnarray*}
and
\begin{eqnarray*}
\mathbb{E}\langle M\rangle_{t}&\leq& L^{2}T+\mathbb{E}\int_{0}^{t}\int_{\mathbb{R}_{0}}|\partial_{x}c(X_{s-},z)+L(\partial_{x}c(X_{s-},z))^{2}|^{2}\nu(dz)ds\nonumber\\
&\leq& (L^{2}+2K+2KL^{2})T.
\end{eqnarray*}
Then
\begin{eqnarray*}
\langle M\rangle_{t}\leq (L^{2}+2K+2KL^{2})T,\quad a.s., 0\leq t\leq T.
\end{eqnarray*}
By Dambis-Dubins-Schwart's theorem~\cite{LG} (see the Theorem 5.13), there exists a unique one-dimensional Brownian motion $\{\beta_{t}\}_{t\geq 0}$ such that $M_{t}=\beta_{\langle M\rangle_{t}}$ and for $0\leq t\leq T$, a.s.
\begin{eqnarray*}
\|D^{B}X_{t}\|^{2}
&\geq&\|\sigma\|_{0}^{2}e^{(-2K-K^{2})T}e^{-2LKT}\cdot t\cdot e^{4\min_{0\leq t\leq(L^{2}+2K+2KL^{2})T}\beta_{t}},
\end{eqnarray*}
and
\begin{eqnarray*}
\mathbb{E}\Big[\frac{1}{\|D^{B}X_{t}\|^{2p}}\Big]\leq \frac{\|\sigma\|_{0}^{-2p}e^{p(2K+K^{2}+L)T}}{t^{p}}\mathbb{E}[e^{4p\max_{0\leq t\leq (L^{2}+2K+2KL^{2})T}(-\beta_{t})}].
\end{eqnarray*}
By the Fernique theorem, we have $\mathbb{E}[e^{4p\max_{0\leq t\leq (L^{2}+2K+2KL^{2})T}(-\beta_{t})}]< \infty$. Carrying out a similar proof and together with (\ref{equ:3.53}), (\ref{equ:b3.53}) can be proved.
\end{proof}
Next we state our main result.
\begin{theorem}
Under assumptions $(\mathbf{H_{1}})$ and $(\mathbf{H_{2}})$, and  $\|\sigma\|_{0}=\\
\inf_{(t,x)\in[0,T]\times\mathbb{R}}|\sigma(t,x)|>0$, then for any $0\leq t\leq T$,
\begin{eqnarray*}\label{MR}
\sup_{x\in\mathbb{R}}|\mathbb{P}(X_{t}^{\epsilon}\leq x)-\mathbb{P}(X_{t}\leq x)|\leq \frac{C}{t}\sqrt{\epsilon},\quad \epsilon\in(0,1).
\end{eqnarray*}
\end{theorem}
\begin{proof}
Let's just prove the case of $D^{N}X_{t}$ and the case of $D^{B}X_{t}$ is similar.
Assume that $\psi$ is a nonnegative smooth function with compact support, let $\phi(y)=\int_{-\infty}^{y}\psi(z)dz$, then by chain rules~\cite{N} (see the Theorem 3.3.2 and Proposition 11.2.2), if $F\in \mathbb{D}^{2}$, then $\phi(F)\in\mathbb{D}^{2}$, and the product of its derivative  $D\phi(F)$ and $D_{r,\xi}^{N}X_{t}$ is given
\begin{eqnarray*}
\langle D^{N}\phi(F),D^{N}X_{t}\rangle_{L^{2}([0,T]\times\mathbb{R}_{0})}=\psi(F)\langle D^{N}F,D^{N}X_{t}\rangle_{L^{2}([0,T]\times\mathbb{R}_{0})},\quad 0\leq t\leq T.
\end{eqnarray*}
In the following, for simplicity, we write $\langle\cdot,\cdot\rangle$ instead of $\langle\cdot,\cdot\rangle_{L^{2}([0,T]\times\mathbb{R}_{0})}$ and $\|\cdot\|$ instead of $\|\cdot\|_{L^{2}([0,T]\times\mathbb{R}_{0})}$.
For fixed $x\in\mathbb{R}$, by a standard approximation procedure, we derive that the above result is valid for $\psi(z)=1_{(-\infty,x]}(z)$. Choose $F=X_{t}^{\epsilon}$ and $F=X_{t}, 0\leq t\leq T$, we have
\begin{eqnarray*}
\Big\langle D^{N}\int_{-\infty}^{X_{t}^{\epsilon}}1_{(-\infty,x]}(z)dz,D^{N}X_{t}\Big\rangle&=&1_{(-\infty,x]}(X_{t}^{\epsilon})\langle D^{N}X_{t}^{\epsilon},D^{N}X_{t}\rangle,\nonumber\\
\Big\langle D^{N}\int_{-\infty}^{X_{t}}1_{(-\infty,x]}(z)dz,D^{N}X_{t}\Big\rangle&=&1_{(-\infty,x]}(X_{t})\langle D^{N}X_{t},D^{N}X_{t}\rangle.
\end{eqnarray*}
Then
\begin{eqnarray}
&&\Big\langle D^{N}\int_{X_{t}}^{X_{t}^{\epsilon}}1_{(-\infty,x]}(z)dz,D^{N}X_{t}\Big\rangle\nonumber\\
&&=1_{(-\infty,x]}(X_{t}^{\epsilon})\langle D^{N}X_{t}^{\epsilon},D^{N}X_{t}\rangle-1_{(-\infty,x]}(X_{t})\langle D^{N}X_{t},D^{N}X_{t}\rangle\nonumber\\
&&=(1_{(-\infty,x]}(X_{t}^{\epsilon})-1_{(-\infty,x]}(X_{t}))\|D^{N}X_{t}\|^{2}+1_{(-\infty,x]}(X_{t}^{\epsilon})\langle D^{N}X_{t}^{\epsilon}-D^{N}X_{t},D^{N}X_{t}\rangle,\nonumber
\end{eqnarray}
and
\begin{eqnarray*}
1_{(-\infty,x]}(X_{t}^{\epsilon})-1_{(-\infty,x]}(X_{t})&=&\frac{\langle D^{N}\int_{X_{t}}^{X_{t}^{\epsilon}}1_{(-\infty,x]}(z)dz,D^{N}X_{t}\rangle}{\|D^{N}X_{t}\|^{2}}\\
&-&\frac{1_{(-\infty,x]}(X_{t}^{\epsilon})\langle D^{N}X_{t}^{\epsilon}-D^{N}X_{t},D^{N}X_{t}\rangle}{\|D^{N}X_{t}\|^{2}},
\end{eqnarray*}
taking expectation,
\begin{eqnarray*}
&&\mathbb{P}(X_{t}^{\epsilon}\leq x)-\mathbb{P}(X_{t}\leq x)\\
&=&\mathbb{E}\Big[\frac{\langle D^{N}\int_{X_{t}}^{X_{t}^{\epsilon}}1_{(-\infty,x]}(z)dz,D^{N}X_{t}\rangle}{\|D^{N}X_{t}\|^{2}}\Big]-\mathbb{E}\Big[\frac{1_{(-\infty,x]}(X_{t}^{\epsilon})\langle D^{N}X_{t}^{\epsilon}-D^{N}X_{t},D^{N}X_{t}\rangle}{\|D^{N}X_{t}\|^{2}}\Big]\\
&=&\mathbb{E}\Big[\int_{X_{t}}^{X_{t}^{\epsilon}}1_{(-\infty,x]}(z)dz\cdot\delta\Big(\frac{D^{N}X_{t}}{\|D^{N}X_{t}\|^{2}}\Big)\Big]-\mathbb{E}\Big[\frac{1_{(-\infty,x]}(X_{t}^{\epsilon})\langle D^{N}X_{t}^{\epsilon}-D^{N}X_{t},D^{N}X_{t}\rangle}{\|D^{N}X_{t}\|^{2}}\Big],
\end{eqnarray*} 
here the second equality is valid by~\cite{N} (see the Proposition 10.2.2). By H${\rm\ddot{o}}$lder's inequality, we have
\begin{eqnarray*}
&&\sup_{x\in\mathbb{R}}|\mathbb{P}(X_{t}^{\epsilon}\leq x)-\mathbb{P}(X_{t}\leq x)|\\
&\leq& \mathbb{E}\Big|(X_{t}^{\epsilon}-X_{t})\delta\Big(\frac{D^{N}X_{t}}{\|D^{N}X_{t}\|^{2}}\Big)\Big|+\mathbb{E}\Big|\frac{1_{(-\infty,x]}(X_{t}^{\epsilon})\langle D^{N}X_{t}^{\epsilon}-D^{N}X_{t},DX_{t}\rangle}{\|D^{N}X_{t}\|^{2}}\Big|\\
&\leq&(\mathbb{E}|X_{t}^{\epsilon}-X_{t}|^{2})^{\frac{1}{2}}\Big(\mathbb{E}\delta\Big(\frac{D^{N}X_{t}}{\|D^{N}X_{t}\|^{2}}\Big)^{2}\Big)^{\frac{1}{2}}+\mathbb{E}\Big|\frac{\|D^{N}X_{t}^{\epsilon}-D^{N}X_{t}\|}{\|D^{N}X_{t}\|}\Big|\\
&\leq&(\mathbb{E}|X_{t}^{\epsilon}-X_{t}|^{2})^{\frac{1}{2}}\Big(\mathbb{E}\delta\Big(\frac{D^{N}X_{t}}{\|D^{N}X_{t}\|^{2}}\Big)^{2}\Big)^{\frac{1}{2}}+(\mathbb{E}\|D^{N}X_{t}^{\epsilon}-D^{N}X_{t}\|^{2})^{\frac{1}{2}}\Big(\mathbb{E}\frac{1}{\|D^{N}X_{t}\|^{2}}\Big)^{\frac{1}{2}}.\\
\end{eqnarray*}
By (\ref{equ:3.34}), we have
\begin{eqnarray*}
\sup_{x\in\mathbb{R}}|\mathbb{P}(X_{t}^{\epsilon}\leq x)-\mathbb{P}(X_{t}\leq x)|&\leq&C\sqrt{\epsilon}\Big[\Big(\mathbb{E}\delta\Big(\frac{D^{N}X_{t}}{\|D^{N}X_{t}\|^{2}}\Big)^{2}\Big)^{\frac{1}{2}}+\Big(\mathbb{E}\frac{1}{\|D^{N}X_{t}\|^{2}}\Big)^{\frac{1}{2}}\Big].
\end{eqnarray*}
Further, by Lemma \ref{lem:3.3},
\begin{eqnarray}
\mathbb{E}\Big[\frac{1}{\|D^{N}X_{t}\|^{2}}\Big]\leq C\frac{1}{t}.
\end{eqnarray}
On the other hand, let $u(r,\xi)=\frac{D_{r,\xi}^{N}X_{t}}{\|D^{N}X_{t}\|^{2}},0\leq r\leq t$, then by~\cite{AL} (see the Proposition 3.22) or~\cite{SUV} (see the Section 6),
\begin{eqnarray*}
&&\mathbb{E}\delta\Big(\frac{D^{N}X_{t}}{\|D^{N}X_{t}\|^{2}}\Big)^{2}=\mathbb{E}[\delta(u(r,\xi))^{2}]\\
&\leq& \int_{0}^{t}\int_{\mathbb{R}_{0}}\mathbb{E}|u(r,\xi)|^{2}\nu(d\xi)dr+\int_{0}^{t}\int_{0}^{t}\int_{\mathbb{R}_{0}}\int_{\mathbb{R}_{0}}\mathbb{E}|D_{s,\eta}^{N}u(r,\xi)|^{2}\nu(d\xi)\nu(d\eta)dsdr\\
&\leq&\mathbb{E}\Big[\frac{1}{\|D^{N}X_{t}\|^{2}}\Big]+\int_{0}^{t}\int_{0}^{t}\int_{\mathbb{R}_{0}}\int_{\mathbb{R}_{0}}\mathbb{E}|D_{s,\eta}^{N}u(r,\xi)|^{2}\nu(d\xi)\nu(d\eta)dsdr\\
&\leq&C t^{-1}+\int_{0}^{t}\int_{0}^{t}\int_{\mathbb{R}_{0}}\int_{\mathbb{R}_{0}}\mathbb{E}|D_{s,\eta}^{N}u(r,\xi)|^{2}\nu(d\xi)\nu(d\eta)dsdr.
\end{eqnarray*}
By the chain rules of Malliavin derivative~\cite{N} (see the Proposition 11.2.2), for $0\leq s\leq t$ and $\eta\in\mathbb{R}_{0}$,
\begin{eqnarray}
D_{s,\eta}^{N}u(r,\xi)=\frac{D_{s,\eta}^{N}D_{r,\xi}^{N}X_{t}}{\|D^{N}X_{t}\|^{2}}-2\frac{D_{r,\xi}^{N}X_{t}\langle D_{r,\xi}^{N}X_{t},D_{s,\eta}^{N}D_{r,\xi}^{N}X_{t}\rangle}{\|D^{N}X_{t}\|^{2}},
\end{eqnarray}
then by the H${\rm \ddot{o}}$lder inequality and Lemma \ref{lem:3.3}, we have
\begin{eqnarray}\label{Dd}
&&\int_{0}^{t}\int_{0}^{t}\int_{\mathbb{R}_{0}}\int_{\mathbb{R}_{0}}\mathbb{E}|D_{s,\eta}^{N}u(r,\xi)|^{2}\nu(d\xi)\nu(d\eta)dsdr\nonumber\\
&\leq& 2\mathbb{E}\frac{\int_{0}^{t}\int_{0}^{t}\int_{\mathbb{R}_{0}}\int_{\mathbb{R}_{0}}
|D_{s,\eta}^{N}D_{r,\xi}^{N}X_{t}|^{2}\nu(d\xi)\nu(d\eta)dsdr}{\|D^{N}X_{t}\|^{4}}+8\mathbb{E}\frac{\int_{0}^{t}\int_{0}^{t}\int_{\mathbb{R}_{0}}\int_{\mathbb{R}_{0}}|D_{s,\eta}^{N}D_{r,\xi}^{N}X_{t}|^{2}\nu(d\xi)\nu(d\eta)dsdr}{\|D^{N}X_{t}\|^{4}}\nonumber\\
&\leq&10\Big(\mathbb{E}\Big|\int_{0}^{t}\int_{0}^{t}\int_{\mathbb{R}_{0}}\int_{\mathbb{R}_{0}}|D_{s,\eta}^{N}D_{r,\xi}^{N}X_{t}|^{2}d\theta dr\Big|^{2}\Big)^{\frac{1}{2}}\Big(\mathbb{E}\Big[\frac{1}{\|D^{N}X_{t}\|^{8}}\Big]\Big)^{\frac{1}{2}}\nonumber\\
&\leq&Ct^{-2}\Big(\int_{0}^{t}\int_{0}^{t}\int_{\mathbb{R}_{0}}\int_{\mathbb{R}_{0}}\mathbb{E}|D_{s,\eta}^{N}D_{r,\xi}^{N}X_{t}|^{4}\nu(d\xi)\nu(d\eta)dsdr\Big)^{\frac{1}{2}}.
\end{eqnarray}
By (\ref{equ:3.52}), for any $0\leq s,r \leq t$ and $\xi, \eta\in \mathbb{R}_{0}$, we have
\begin{eqnarray*}
&&D_{s,\eta}^{N}D_{r,\xi}^{N}X_{t}\\
&&=\partial_{x}\partial_{z}c(X_{r},\xi)D_{s,\eta}^{N}X_{r}+\int_{s\vee r}^{t}[b''(v,X_{v})D_{r,\xi}^{N}X_{v}D_{s,\eta}^{N}X_{v}+b'(v,X_{v})D_{s,\eta}^{N}D_{r,\xi}^{N}X_{v}]dv\\
&&+\int_{s\vee r}^{t}[\sigma''(v,X_{v})D_{r,\xi}^{N}X_{v}D_{s,\eta}^{N}X_{v}+\sigma'(v,X_{v})D_{s,\eta}^{N}D_{r,\xi}^{N}X_{v}]dB_{v}+\partial_{z}\partial_{x}c(X_{s},\eta)D_{r,\xi}^{N}X_{s}\\
&&+\int_{s\vee r}^{t}\int_{\mathbb{R}_{0}}[\partial^{2}_{x}c(X_{v-},z)D_{s,\eta}^{N}X_{v}D_{r,\xi}^{N}X_{v}+\partial_{x}c(X_{v-},z)D_{r,\xi}^{N}D_{s,\eta}^{N}X_{v-}]\tilde{N}(dv,dz)\\
\end{eqnarray*}
Note that for any $p\geq2$, by a similar proof to (\ref{equ:3.33}), we have $\mathbb{E}|D_{r,\xi}^{N}X_{t}|^{p}\leq C, 0\leq r\leq t\leq T$. By the boundedness of $b',b'',\sigma',\sigma''$ and $(\mathbf{H_{2}})$,
\begin{eqnarray}
\mathbb{E}|D_{s,\eta}^{N}D_{r,\xi}^{N}X_{t}|^{4}\leq C+\int_{s\vee r}^{t}\mathbb{E}|D_{s,\eta}^{N}D_{r,\xi}^{N}X_{v}|^{4}dv,
\end{eqnarray}
Gronwall's inequality yields
\begin{eqnarray}
\mathbb{E}|D_{s,\eta}^{N}D_{r,\xi}^{N}X_{t}|^{4}\leq C_{T}, \quad0\leq s,r\leq t\leq T.
\end{eqnarray}
Then by (\ref{Dd}), we have
\begin{eqnarray}
\int_{0}^{t}\int_{0}^{t}\int_{\mathbb{R}_{0}}\int_{\mathbb{R}_{0}}\mathbb{E}|D_{s,\eta}^{N}u(r,\xi)|^{2}\nu(d\xi)\nu(d\eta)dsdr\leq Ct^{-2}.
\end{eqnarray}
and 
\begin{eqnarray*}
\sup_{x\in\mathbb{R}}|\mathbb{P}(X_{t}^{\epsilon}\leq x)-\mathbb{P}(X_{t}\leq x)|&\leq&C\sqrt{\epsilon}(t^{-\frac{1}{2}}+t^{-1})\leq Ct^{-1}\sqrt{\epsilon}.
\end{eqnarray*}
\end{proof}

\end{document}